\documentclass[12pt]{amsart}
\usepackage{amsfonts}
\usepackage{amscd}
\usepackage{amssymb}
\usepackage{graphics}

\usepackage{amsmath}

\usepackage{hyperref}
\hypersetup{colorlinks,citecolor=blue}

\newcommand{\BR}{{\mathbb{R}}}
\newcommand{\BC}{{\mathbb{C}}}
\newcommand{\BF}{{\mathbb{F}}}
\newcommand{\BQ}{{\mathbb{Q}}}

\newcommand{\BG}{{\mathbb{G}}}

\newcommand{\BH}{{\mathbb{H}}}

\newcommand{\gC}{\Gamma}

\newcommand{\gs}{\sigma}

\newcommand{\gO}{\Omega}

\newcommand{\ga}{\alpha}
\newcommand{\gt}{\tau}

\newcommand{\ti}[1]{\tilde{#1}}

\newcommand{\aut}{\text{Aut}}
\newcommand{\out}{\text{Out}}
\newcommand{\Hom}{\text{Hom}}

\newcommand{\Ad}{\text{Ad}}

\newcommand{\SL}{\text{SL}}

\def\Aut{\text{Aut}}
\def\Aut{\text{Out}}

\newtheorem{prop}{Proposition}[section]
\newtheorem{thm}[prop]{Theorem}
\newtheorem{lem}[prop]{Lemma}
\newtheorem{cor}[prop]{Corollary}

\theoremstyle{definition}

\newtheorem{defn}[prop]{Definition}
\newtheorem{rem}[prop]{Remark}

\newtheorem{prob}[prop]{Problem}

\newcommand\PS{{\mathcal{PS}}}
\newcommand\PP{{\mathcal{P}}}
\newcommand\AAA{{\mathcal{A}}}
\newcommand\RR{{\mathcal{R}}}

\newcommand\union{\cup}

\catcode`\@=11
\long\def\@savemarbox#1#2{\global\setbox#1\vtop{\hsize\marginparwidth 
  \@parboxrestore\tiny\raggedright #2}}
\catcode`\@=12

\begin{document}
\author{Tsachik Gelander and Yair Minsky}

\thanks{Tsachik Gelander acknowledges the financial support from the European
Research Council (ERC)/ grant agreement 203418, and the Israeli Science
Foundation grant 1345/07.
Yair Minsky acknowledges support from the National Science
Foundation. 
}

\date{\today}

\title{The dynamics of $\aut(F_n)$ on redundant representations}
\maketitle

\begin{abstract}
We study some dynamical properties of the canonical $\aut(F_n)$-action
on the space $\mathcal{R}_n(G)$ of redundant representations of the
free group $F_n$ in $G$, where $G$ is the group of rational points of
a  simple algebraic group over a local field.  We show that this
action is always minimal and ergodic, confirming a conjecture of
A. Lubotzky. On the other hand for the classical cases where
$G=\SL_2(\BR)$ or $\SL_2(\BC)$ we show that the action is not weak
mixing, in the sense that the diagonal action on $\mathcal{R}_n(G)^2$
is not ergodic. 
\end{abstract}

\section{A short introduction}

Let $G$ be a group and consider $\Hom(F_n,G)$ --- the representation
space of the free group $F_n$ in $G$. The automorphism group
$\aut(F_n)$ acts naturally on $\Hom(F_n,G)$ by pre-compositions,
inducing a canonical action of $\Aut(F_n)$ on
$\chi_n(G)=\Hom(F_n,G)/G$ --- the space of conjugacy classes of
representations. (The difference between this and the character
variety $\Hom(F_n,G)//G$ will not be important to us in this paper,
and in particular they agree on the set of irreducible
representations.)
When $G$ has an additional structure (e.g. $G$ is algebraic, topological, measurable or finite, etc.) $\Hom(F_n,G)$ often inherits the structure from $G$ and the action respects the structure. For instance, if $G$ is a topological group $\Hom(F_n,G)$ is a topological space and $\aut(F_n)$ acts by homeomorphisms. Similarly, if $G$ is a locally compact group, the Haar measure induces a measure on $\Hom(F_n,G)$ and $\aut(F_n)$ preserves its measure class. Moreover, if $G$ is unimodular, $\aut(F_n)$ preserves the measure. There are various reasons why people are interested in understanding the invariant subsets, and more generally the dynamics, of this action (we refer to Lubotzky's survey \cite{Lub} for some of the motivations). 

W.M. Goldman conjectured that for every compact connected Lie group $G$, if $n\ge 3$, the $\aut(F_n)$ action on $\Hom(F_n,G)$ is ergodic. As he pointed out for $n=2$ the action is not ergodic in general since the function $f\mapsto \text{trace}(f(xyx^{-1}y^{-1}))$, where $x,y\in F_2$ are free generators, is $\aut(F_2)$-invariant and nonconstant if $G$ is noncommutative. In \cite{Goldman}, Goldman proved his conjecture for the case that all the simple factors of $G$ are locally isomorphic to $\text{SU}(2)$. The general case of Goldman's conjecture was proved later in \cite{Gelander}. 

The compact-connected case is a bit misleading. For general $G$ one first restricts the attention to the subspace of epimorphisms, where in the context of topological groups, by epimorphism we mean a homomorphism with dense image: 
$$
 \text{Epi}(F_n,G)=\{f\in\Hom(F_n,G):\overline{f(F_n)}=G\}.
$$
This is a measurable invariant subset, often even open, but in general its complement is also big and can be divided to subsets of the form $\text{Epi}(F_n,H)$ for closed subgroups $H\le G$. In the special case when $G$ is connected and compact almost every homomorphism has dense image (see \cite[Lemma 1.10]{Gelander}) and $\Hom(F_n,G)=\text{Epi}(F_n,G)$ as measure spaces.

The first noncompact cases that were studied are $G=\SL_2(k)$ where
$k$ is a local field. The mutual outcomes were somewhat surprising.
More or less simultaneously, Y. Glasner \cite{Glasner} showed that
when $k$ is nonarchimedean $\aut(F_{n\ge 3})$ acts ergodically on
$\text{Epi}(F_n,G)$\footnote{Assuming $\text{char}(k)\ne 2$}, while
Y. Minsky \cite{minsky:primitive} showed that for $k=\BR,\BC$ the
action is not ergodic. Minsky \cite{minsky:primitive} defined the
notion of primitive-stable homomorphism, and proved that the set
of primitive-stable representations is open, containing
the Schottky representations as well as a part of $\text{Epi}(F_n,G)$,
and the action of $Out(F_n)$ is properly discontinuous
on the set $\mathcal{PS}(F_n,G)$ of conjugacy classes of 
primitive-stable representations.
On the other hand, in the nonarchimedean case, as one
can deduce from Weidmann's theorem \cite{Weidmann,Glasner}, there are
no primitive stable representations of $F_n$ in
$\SL_2(k)$\footnote{Kapovich and Weidmann \cite{KW} established a kind
  of generalization of Weidmann's theorem which applies in particular
  for $\SL_2(\BR,\BC)$. It might be interesting to investigate the
  interplay between Minsky's and Kapovich--Weidmann's results.}.

In an attempt to understand the global picture, and partly motivated by analogous results from finite group theory, A. Lubotzky \cite{Lub} formulated the correct conjecture, namely that the action on the {\it big} subset of redundant representations is {\it always} ergodic. Recall:

\begin{defn}
A representation $\rho:F_n\to G$ is redundant if there exists a proper free factor $A$ of $F_n$ with $\rho(A)$ dense in $G$. We denote by $\mathcal{R}_n(G)\subset \Hom(F_n,G)$ the set of redundant representations.
\end{defn} 

When $G$ is a simple Lie group over a local field, the set $\mathcal{R}_n(G)$ is open (see Corollary \ref{cor:R-open}). 

At first glance, Lubotzky's conjecture may seem wrong for the following reason. Note that a representation $\rho$ is redundant iff there is a free generating set $\{x_1,\ldots,x_n\}$ for $F_n$ such that $\rho(\langle x_1,\ldots,x_{n-1}\rangle)$ is dense in $G$. Call a representation $\rho:F_n\to G$ {\it very redundant} if for {\it any} free generating set $\{ x_1,\ldots,x_n\}$ for $F_n$ and every $1\le i\le n$, $\rho(\langle x_j:j\ne i\rangle)$ is dense in $G$ and let $\mathcal{VR}_n(G)$ be the set of very redundant representations. Clearly $\mathcal{VR}_n(G)$ is measurable, $\aut(F_n)$-invariant and strictly contained in $\mathcal{R}_n(G)$. Hence if both $\mathcal{VR}_n(G)$ and $\mathcal{R}_n(G)\setminus\mathcal{VR}_n(G)$ have positive measure, the conjecture is false. 
However, while when $G$ is compact and $n\ge 3$, almost every representation is very redundant, when $G$ is noncompact one can show that there are no very redundant representations at all. Moreover, Lubotzky's conjecture is indeed true (see Theorem \ref{thm1} below).

\section{Statement of the main results}

The following theorem confirms Lubotzky's conjecture:

\begin{thm}\label{thm1}
Let $k$ be a local field, $\BG$ a Zariski connected simply connected\footnote{When $k=\BR,\BC$ the simply connectedness assumption is unnecessary.} simple
$k$ group, and $G=\BG(k)$ the group of $k$ points. If $\text{char}(k)>0$ assume further that the adjoint representation of $G$ is irreducible.
Then the action of $\aut(F_n)$ on $\mathcal{R}_n(G)$ is ergodic with respect to the Haar measure induced from $\Hom(F_n,G)\cong G^n$.
\end{thm}

In case $G$ is compact and connected almost every representation of $F_{n\ge 3}$ is redundant, hence Theorem \ref{thm1} recovers Theorem 1.6 of \cite{Gelander}, namely that the action of $\text{Aut}(F_n)$ on the representation variety $Hom(F_n,G)$ is ergodic. 
Similarly, when $k$ is nonarchimedean and $G=\SL_2(k)$, almost every representation of $F_{n\ge 3}$ with dense image is redundant, as Glasner showed using Weidmann's theorem (see \cite{Glasner} for details). Hence the main result of 
\cite{Glasner}\footnote{In \cite{Glasner} also the case of $G=\text{Aut}(T)$, where $T$ is a regular tree, was treated. This case is not covered by \ref{thm1}.} is also a special case of Theorem \ref{thm1}.  

When $G$ is compact and $n\ge 3$, the action
$\aut(F_n)\propto\mathcal{R}_n(G)$ is even weakly mixing. This is
because $\mathcal{R}_n(G)=\text{Hom}(F_n,G)$ (up to measure 0) and
$\text{Hom}(F_n,G)\times \text{Hom}(F_n,G)$ is canonically identified
with $\text{Hom}(F_n,G\times G)$ while $G\times G$ is again a compact
connected group (see \cite{Gelander} for more details). Somehow,
perhaps unexpectedly, this stronger property of compact groups does
not hold in general: 

\begin{thm}\label{thm2}
Let $G=\SL_2(\BR)$ or $\SL_2(\BC)$, and $n\ge 3$. Then the action of
$\text{Out}(F_n)$ on $\overline\RR_n(G)$ is not weakly mixing,
in the sense that the diagonal action of $Out(F_n)$ 
on $\overline\RR_n(G)\times\overline\RR_n(G)$ is not ergodic.
\end{thm}  
(Here $\overline\RR_n(G)$ is the image of $\RR_n(G)$ in 
$\chi_n(G)$. Note that nonergodicity in the quotient implies
it for the action of $\aut(F_n)$ upstairs). 

\begin{rem}
We state and prove Theorem \ref{thm2} for $\SL_2$ since we will use $3$ dimensional hyperbolic geometry in the proof. However the result extends immediately to every rank one simple Lie group $G$ (see also Remark \ref{rem:NMr1}).
\end{rem}

Recall that an action on a topological space is minimal if every orbit is dense. The representation space $\text{Hom}(F_n,G)$, hence also its subspace $\mathcal{R}_n(G)$, inherits a canonical topology from $G$.
Moreover, the set $\mathcal{R}_n(G)$ is open  (cf. Corollary \ref{cor:R-open} below).
The following result is new even in the context of compact Lie groups (although a hint for it for compact $G$ is given in \cite[Remark 1.5(1)]{Gelander}).

\begin{thm}\label{thm3}
Let $G$ be as in Theorem \ref{thm1}. The action of $\text{Aut}(F_n)$ on $\mathcal{R}_n(G)$ is minimal.
\end{thm}

\begin{rem}\label{rem:semisimple}
Theorems \ref{thm1} and \ref{thm3} remain true, and the proofs requires only minor changes, when $G$ is a general connected semisimple Lie group (not necessarily simple or linear algebraic). However, when $k$ is nonarchimedean and $G$ has more than one factor (i.e. semisimple but not simple) some parts of our arguments cannot be applied directly.
\end{rem}

In analogy, when $G$ is a finite simple group, a classical result of
Gilman \cite{Gilman} (for $n\ge 4$) and Evans \cite{Evans} (for $n\ge
3$) states that $\aut(F_n)$ acts transitively on
$\mathcal{R}_n(G)$. As a consequence, the well known Weigold
conjecture that $\aut(F_n)$ acts transitively on $\text{Epi}(F_n,G)$,
or equivalently, that the associated product replacement graph is
connected, reduces to the conjecture that every epimorphism is
redundant, i.e. that $\text{Epi}(F_n,G)=\mathcal{R}_n(G)$. Theorems
\ref{thm1} and \ref{thm3} can be thought of as locally compact analogs
of the Gilman--Evans theorem where instead of groups such as $\SL_n(\BF_q)$ we consider $\SL_n(\BR)$ and $\SL_n(\BQ_p)$.

%--------------------------------------------------------------------------

\section{Remarks about dense subgroups}\label{sec:dense}
In this section we form some basic results about dense subgroups that are relevant in the proofs of the main results.

Let $k$ be a local field (i.e. $\BR,\BC$, a finite extension of $\BQ_p$ for some rational prime $p$, or the field $\BF_q((t))$ of Laurent series over a finite field), $\BG$ a Zariski connected simple algebraic group defined over $k$ and $G=\BG(k)$ the group of $k$ rational points. In case $k$ is archimedean (i.e. $\BR$ or $\BC$), $G$ is a connected real analytic Lie group, and in case $k$ is a finite extension of $\BQ_p$, $G$ is a $p$-adic analytic group. In the non-archimedean case, we also suppose that $\BG$ is simply connected. We denote by $\mathfrak{g}$ the Lie algebra of $G$, and by $\mathcal{A}$ the simple associative algebra spanned by the image of the adjoint representation $\Ad:G\to\aut(\mathfrak{g})$. In the positive characteristic case, it is not always true that the representation $\text{Ad}$ is irreducible, but we will restrict ourselves to that case, thus by Burnside's theorem $\mathcal{A}=\text{End}(\mathfrak{g})$.

Let us first formulate some simple useful criterions for a subgroup of $G$ to be dense:  

\medskip

\noindent{\bf An archimedean density criterion ($k=\BR$ or $\BC$)}: 
{\it A subgroup $\gC\le G$ is dense iff it is nondiscrete and $\Ad(\gC)$ generates $\mathcal{A}$}.

\medskip

The implication $\Rightarrow$ is obvious. For the other direction,
denote $H=\overline{\gC}$  and $\mathfrak{h}=\text{Lie}(H)$ its Lie algebra. One sees that $\mathfrak{h}$ is an ideal of $\mathfrak{g}$ (being $\Ad(\gC)$ invariant) of positive dimension (since $H$ is nondiscrete). As $\mathfrak{g}$ is simple, it follows that $\mathfrak{h}=\mathfrak{g}$ and hence $H=G$.
\qed

\medskip 

\noindent{\bf A nonarchimedean density criterion ($k$ is totally disconnected)}: 
{\it A subgroup $\gC\le G$ is dense iff it is nondiscrete, unbounded, $\Ad(\gC)$ generates $\mathcal{A}$, and the entries of $\gC$ are not contained in a proper closed subfield of $k$}.

\medskip

To explain the nontrivial implication $\Leftarrow$, let us again
denote $H=\overline{\gC}$  and $\mathfrak{h}=\text{Lie}(H)$. As in the
archimedean case, $\mathfrak{h}$ is the full Lie algebra of $G$. It
follows that $\dim H=\dim G$ and hence $H$ is Zariski dense. We claim furthermore that $H$ is open. This is a consequence of the following criterion of R. Pink \cite{Pink}:

\begin{lem}[Pink's criterion]
A compact subgroup of $G$ is open iff it is Zariski dense and not contained in $\BG(k')$ for a proper closed subfield $k'$.
\end{lem}

Let $U$ be an open compact subgroup of $G$ and consider the compact group $H\cap U$. It is well known that $H\cap U$ is Zariski dense in $H$ and $\text{Lie}(H\cap U)=\text{Lie}(H)$ (see \cite[Lemma 3.2]{Pla-Rap}). Moreover since $H\nleqslant\BG(k')$ for every closed subfield $k'<k$ while the adjoint representation is defined over the prime field, we deduce that $\Ad(H)$ is not contained in $\Ad(\BG)(k')$, and since, for $h\in H$, $\Ad(h)$ is determined by the restriction of $i_h:g\mapsto hgh^{-1}$ to the open (Zariski dense) set $H\cap U\cap h^{-1}Uh$, we deduce that $H\cap U\nleqslant\BG(k')$ (see \cite{Glasner} or \cite{Shalom} for more details). We deduce from Pink's criterion that $H$ is open.

Finally, the density criterion follows from the following result of Tits \cite{Prasad}:

\begin{lem}\label{lem:Howe-Moore}
If $H\le G$ is open and unbounded then $H=G$.
\end{lem}

For the reader's convenience we include a proof of Lemma
\ref{lem:Howe-Moore} (we believe this proof appears somewhere in the
literature, but we are not aware of the correct source). Consider the
unitary representation of $G$ on the separable Hilbert space
$l_2(G/H)$ (or $l_2(G/H)^0$ if $|G/H|<\infty$) arising from the left
action of $G$ on $G/H$. Clearly, there are no nonzero invariant
vectors. However, if $[G:H]>1$, the unbounded subgroup $H$ admits a nontrivial invariant unit vector, in contrast to the Howe-Moore theorem. Hence $G=H$.
\qed

\medskip

Here is another basic result:

\begin{prop}
The set $\text{Epi}(F_n,G)=\{ f\in\Hom(F_n,G): f(F_n)~\text{is dense in}~G\}$ is open in $\Hom(F_n,G)$, and nonempty provided $n\ge 2$.
\end{prop}

\begin{proof}
This well known when $k$ is archimedean (see \cite{Ku1,GZ,BG1}). 

Suppose that $k$ is nonarchimedean, and let $f:F_n\to G$ be a homomorphism with dense image.
Since the set of nonelliptic elements in $G$ is open,
 $f(F_n)$, as well as $f'(F_n)$ for any $f'\in\text{Hom}(F_n,G)$ sufficiently close to $f$, contains a nonelliptic element and is hence unbounded. Moreover, $G$ admits an open finitely generated pro-$p$ group $K$ (see \cite{Bar-Lar}). It follows that the Frattini subgroup $F$ of $K$ is open. A subgroup of $K$ is dense in $K$ iff it intersects each of the finitely many open cosets of $F$ in $K$. This is clearly an open condition. This shows that $\overline{f'(F_n)}$ for any $f'$ sufficiently close to $f$ is open and unbounded. By Lemma \ref{lem:Howe-Moore}, any such $f'$ has dense image.
Hence $\text{Epi}(F_n,G)$ is open.

To show the second statement, we have to produce a $2$ generated dense subgroup of $G$. 
First note that since the associative algebra $\mathcal{A}$ is finite dimensional, the set 
$$
 \{(a,b):\Ad(a),\Ad(b)~\text{generates}~\mathcal{A}\}
$$
is Zariski open in $G^2$, and since $G$ admits a $2$-generated open subgroup (see \cite{Bar-Lar})
it is nonempty.
Pick $(a,b)$ in this set such that $a$ is elliptic of infinite order and $b$ is nonelliptic. The closed field $k'$ generated by the entries of $a$ is a local subfield of $k$. There are only finitely many intermediate fields between $k'$ and $k$, hence, slightly deforming $b$ if necessary, we may assume that its entries are not contained in any of these intermediate fields. By the nonarchimedean density criterion sited above, $\langle a,b\rangle$ is dense. 
\end{proof}

As an immediate corollary we have:

\begin{cor}\label{cor:R-open}
The set $\mathcal{R}_n(G)$ of redundant representations of $F_n$ in $G$ is also open in $\Hom(F_n,G)$. Moreover $\mathcal{R}_n(G)\ne\emptyset$ provided $n\ge 3$.
\end{cor}

For a finite collection of elements $g_1,\ldots,g_k\in G$ we define $\gO(g_1,\ldots,g_k)$ to be the set of elements $g$ in $G$ that together with $g_1,\ldots,g_k$ generates a dense subgroup of $G$:
$$
 \gO(g_1,\ldots,g_k):=\{g\in G:\langle g_1,\ldots,g_k,g\rangle~\text{is dense in}~G\}.
$$
Then $\gO(g_1,\ldots,g_k)$ is an open (possibly empty) subset of $G$.
We will sometimes abuse notation and write $\gO(S)$ for $\gO(g_1,\ldots,g_k)$ when $S$ is the set $S=\{g_1,\ldots,g_k\})$.
We will need the data that (under certain conditions) sets of this form intersect each other. In the archimedean case this will follow from:

\begin{lem}
Suppose $k$ is archimedean. If $\gO(g_1,\ldots,g_k)$ is nonempty, then there is an identity neighborhood $U$ in $G$, and a proper algebraic subvariety $X\subset G$ such that $\gO(g_1,\ldots,g_k)$ contains $U\setminus X$.
\end{lem}

\begin{proof}
If $\gO(g_1,\ldots,g_k)$ is nonempty, picking $g\in \gO(g_1,\ldots,g_k)$ we may find finitely many words $W_i,~i=1,\ldots,m$ of $k+1$ variables such that $\Ad(W_i(g_1,\ldots,g_k,g)),~i=1,\ldots,m$ spans $\mathcal{A}$ as a vector space. Define 
$$
 X=\{g: \text{span}\{\Ad(W_i(g_1,\ldots,g_k,g)),~i=1,\ldots,m\} \neq\mathcal{A}\}.
$$
Let $V$ be a relatively compact Zassenhaus identity neighborhood of $G$ (see \cite[Chapter 8]{Raghunathan}). Recall that for every finite subset $\mathcal{S}\subset V$ which generates a discrete subgroup, the Lie algebra generated by $\{\log s:s\in \mathcal{S}\}$ is nilpotent. Let $U$ be a sufficiently small identity neighborhood in $G$ such that for any $u\in U,v\in V$ and every word $\mathcal{W}$ in $m$ letters of length $\le \dim G+1$ we have
$$
 \mathcal{W}(W_1(g_1,\ldots,g_k,v),\ldots,W_m(g_1,\ldots,g_k,v))u\mathcal{W}(W_1(g_1,\ldots,g_k,v),\ldots,W_m(g_1,\ldots,g_k,v))^{-1}\in V.
$$
Now if $g\in U\setminus X$ then $\{\Ad(W_i(g_1,\ldots,g_k,g)),~i=1,\ldots,m\}$ generates $\mathcal{A}$. Thus, by the archimedean density criterion, in order to prove that $\langle g_1,\ldots,g_k,g\rangle$ is dense, it is enough to show that it is nondiscrete. Suppose in contrary that it is discrete. Then for every $j\le\dim G+1$ the Lie algebra 
$$
\mathfrak{n}_j= \langle \log (\mathcal{W} g\mathcal{W}^{-1}): \mathcal{W} ~\text{is a word in}~W_i(g_1,\ldots,g_k,g)~\text{of length}~\le j\rangle
$$ 
is nilpotent. But then for some $j\le\dim G$ we have $\mathfrak{n}_j=\mathfrak{n}_{j+1}$ which forces the nontrivial nilpotent Lie algebra $\mathfrak{n}_j$ to be an ideal, since $\Ad(W_i(g_1,\ldots,g_k,g))$ generates $\mathcal{A}$. A contradiction to the simplicity of $\mathfrak{g}$.
\end{proof}

In particular the collection of these sets have the finite intersection property:

\begin{cor}\label{cor:arch-inter}
In the archimedean case, every finite collection of nonempty sets of the form $\gO(g_1,\ldots,g_k)$ have a nonempty intersection.
\end{cor}

In the nonarchimedean case we prove a somewhat weaker result:

\begin{lem}\label{lem:nonarch-inter1}
Let $S_j,~j=1,\ldots,r$ be a finite family of finite sets. Assume that $\gO(S_j)$ is nonempty for every $j\le r$ and that the groups $\langle S_j\rangle$ are simultaneously all nondiscrete or unbounded. Then $\cap_{j\le r}\gO(S_j)\ne\emptyset$.
\end{lem}

\begin{proof}
The fact that $\gO(S_j)\ne\emptyset$ implies that for all $g$ outside some proper algebraic subvariety $X_j$ the elements $\Ad (s),~s\in S_j\cup\{ g\}$ generate the algebra $\mathcal{A}$. Let $k_j\le k$ be the closed subfield of $k$ generated by the entries of the elements of $S_j$, and let $\{k_{j,i}\}_{i=1}^{n_j}$ be the finite collection of proper local subfields in $k$ containing $k_j$ (if $k_j=k$ this collection is empty).
If all the $\langle S_j\rangle$ are nondiscrete (resp. unbounded) pick 
$$
 g\in G\setminus\big(\cup_{j\le r} X_j\bigcup \cup_{j\le r,i\le n_j}\BG(k_{j,i})\big)
$$ 
nonelliptic (resp. elliptic of infinite order). Then each of the groups $\langle s:s\in S_j\cup\{ g\}\rangle$ satisfies the four condition of the nonarchimedean density criterion, i.e. it
is unbounded, nondiscrete, its image under $\Ad$ generate $\mathcal{A}$ and its entries generate $k$. 
\end{proof}

We will also need:

\begin{lem}\label{lem:2-intersect}
Suppose $S_i,~i=1,2$ are finite sets such that for both $i$, $\gO(S_i)\ne\emptyset$ and each $S_i$ contains a nontorsion element. Then $\gO(S_1)\cap\gO(S_2)\ne\emptyset$.
\end{lem}

\begin{proof}
Let $s_i\in S_i$ be a nontorsion element.
In view of the previous lemma, it suffices to consider the case where $s_1$ is elliptic and $s_2$ is not. Then one deduces that there is an open set $V_2\subset\gO(S_2)$ of elliptic elements. Moreover by choosing $V_2$ to be inside a small neighborhood of the identity, we can guarantee that the set $s_2V_2$ consists of nonelliptic elements. Let $X_1$ be the proper algebraic subvariety 
$$
 X_1=\{g\in G:\Ad(g),\Ad(s),~s\in S_1~\text{do not generate}~\mathcal{A}\},
$$ 
and let $k_1,\ldots,k_n$ be the proper local subfields containing the field generated by the entries of $S_1$.
Then 
$$
 s_2V_2\setminus \big( X_1\bigcup\cup_{i=1}^n\BG(k_i)\big)\subset\gO(S_1)\cap\gO(S_2).
$$ 
Indeed, if $g\in s_2V_2\setminus \big( X_1\bigcup\cup_{i=1}^n\BG(k_i)\big)$ then $g$ together with $S_1$ generates an unbounded (since $g$ is nonelliptic) nondiscrete (since $s_1$ is elliptic of infinite order) subgroup whose image under $\Ad$ generates $\mathcal{A}$, and is not contained in $\BG(k')$ for a proper local subfield $k'<k$, 
hence is dense. On the other hand, $g$ together with $S_2$ generates a subgroup which contain $S_2\cup\{ s_2^{-1}g\}$ and is hence dense.
\end{proof}

For a finite set $S=\{g_1,\ldots,g_k\}$ let us also define:
$$
 \ti\gO(S):=\ti\gO(g_1,\ldots,g_k):=\bigcap_{i=1,\dots,k}\gO(S\setminus\{g_i\}).
$$

We will say that an ordered set (or an $n$-tuple) $S=(g_1,\ldots,g_n)\subset G^n$ is redundant if the element $f\in\Hom(F_n,G)$ defined by $f(x_i)=g_i$, where $\{ x_1,\ldots,x_n\}$ is an arbitrary base, is redundant (this is independent of the choice of the generators $x_i$). For $\gs\in\aut(F_n)$ we will denote by $\gs\cdot S$ the ordered set $(f(\gs^{-1}\cdot x_1),\ldots,f(\gs^{-1}\cdot x_n))$.
We will make use of the following:

\begin{lem}\label{lem:tiO-nonempty}
Let $S$ be an ordered set of size $n$ in $G$. Suppose that either
\begin{itemize}
\item $S$ is redundant, or
\item $\langle S\rangle$ is dense and $S$ contains two nontorsion elements which are simultaneously elliptic or nonelliptic.
\end{itemize}
Then there is $\gs\in\aut(F_n)$ such that $\ti\gO(\gs\cdot S)\ne\emptyset$.
\end{lem}

\begin{proof}
When $k$ is archimedean the lemma follows directly from Corollary \ref{cor:arch-inter} with $\gs=1$, even if we only assume that $\langle S\rangle$ is dense. 

If $k$ is nonarchimedean and the second condition holds, $\ti\gO(S)\ne\emptyset$ by Lemma \ref{lem:nonarch-inter1}.

Assume therefore that $k$ is nonarchimedean and $S$ is redundant. Up to replacing $S$ by $\gs\cdot S$ for a suitable $\gs\in\aut(F_n)$ we may assume that $S=(g_1,\ldots,g_n)$ and $\langle g_1,\ldots,g_{n-1}\rangle$ is dense in $G$. Then the open set $\gO(g_2,\ldots,g_{n-1})$ is nonempty as it contains $g_1$, and hence we may multiply $g_n$ by some element $g'$ belonging to the dense subgroup $\langle g_1,\ldots,g_{n-1}\rangle$ and obtain a nonelliptic element $g'g_n$ belonging to $\gO(g_2,\ldots,g_{n-1})$. Then we can find an element $g''$ belonging to the dense subgroup $\langle g_2,\ldots,g'g_n\rangle$ such that $g''g_1$ is nonelliptic and belongs to the nonempty open set $\gO(g_2,\ldots,g_{n-1})$. Note that the ordered set $S'=(g''g_1,g_2,\ldots,g_{n-1},g'g_n)$ was obtained from $S$ by a sequence of Nielsen transformations and is hence of the form $\gt\cdot S$ for some $\gt\in\aut(F_n)$. Moreover, any subset of cardinality $n-1$ of $S'$ contains either the first or the last element (which are both nonelliptic). Hence by Lemma \ref{lem:nonarch-inter1} $\ti\gO(S')\ne\emptyset$.
\end{proof}

%--------------------------------------------------------------------------

\section{Minimality}
In this section we prove Theorem \ref{thm3}.

Given an element $\phi\in\mathcal{R}_n(G)$ and an open set $U\subset\mathcal{R}_n(G)$ we will find $\ga\in\aut(F_n)$ with $\ga\cdot \phi\in U$. By the definition of $\mathcal{R}_n(G)$, for an appropriate free generating set $\{x_1,\ldots,x_n\}$ we have that $\langle \phi(x_i):i\le n-1\rangle$ is dense in $G$. Moreover acting by Nielsen transformations which change only the last coordinate, and then by Nielsen transformations which change only the first coordinate, we may change $\phi$ so that in addition to the previous condition, $\phi(x_n)\in\gO(\phi(x_2),\ldots,\phi(x_{n-1}))$ and $\phi(x_1)$ is nontorsion. 
Moving $U$ by some appropriate element of $\aut(F_n)$ we may furthermore assume that for some $\phi'\in U$, $\langle \phi'(x_i):i\le n-1\rangle$ is dense, and $\phi'(x_1)$ is nontorsion as well.

We will say that an element $\psi\in\mathcal{R}_n(G)$ {\it links} an element $\varphi\in\mathcal{R}_n(G)$ if for every $k<n$, the group 
$$
 \langle \varphi(x_1),\ldots,\varphi(x_{k-1}),\psi(x_{k+1}),\ldots,\psi(x_n)\rangle
$$ 
is dense in $G$. The set 
$$
 \L(\varphi):=\{\psi\in\mathcal{R}_n(G):\psi~\text{links}~\varphi\}
$$ 
is always open. 

We claim that $\L(\phi)$ is contained in the closure of the orbit
$\aut(F_n)\cdot \phi$ (and the analog statement for $\phi'$). Indeed,
given $\psi\in \L(\phi)$, since $\langle \phi(x_i):i<n\rangle$ is
dense and $\psi(x_n)$ belongs, by definition, to
$\gO(\phi(x_1),\ldots,\phi(x_{n-2}))$, for an appropriate composition
of Nielsen transformations which act on the $n$-th coordinate by
multiplying it with other coordinates, we obtain an element $\gs_n$
for which $\gs_n\cdot \phi(x_i)=\phi(x_i)$ for $i<n$ and $\gs_n\cdot
\phi(x_n)$ is arbitrarily close to $\psi(x_n)$ and belongs to
$\gO(\phi(x_1),\ldots,\phi(x_{n-2}))$. After that, using the density
of $\langle\gs_n\cdot\phi
(x_1),\ldots,\gs_n\cdot\phi(x_{n-2}),\gs_n\cdot \phi(x_n)\rangle$ we
may find an element $\gs_{n-1}\in\aut(F_n)$ which is a composition of
Nielsen transformations acting on the $(n-1)$-th coordinate by
multiplying it by the others, such that $\gs_{n-1}\gs_n\cdot
\phi(x_i)=\gs_n\phi(x_i)$ for $i\ne n-1$, and $\gs_{n-1}\gs_n\cdot
\phi(x_{n-1})$ belongs to
$\gO(\phi(x_1),\ldots,\phi(x_{n-3}),\gs_n\cdot \phi(x_n))$ and is
arbitrarily close to $\psi(x_{n-1})$. Repeating this procedure
recursively for the lower indices we obtain an element
$\gs_1\gs_2\ldots\gs_n$ which moves $\phi$ arbitrarily close to
$\psi$.  

Next observe that $\L(\phi)\cap\L(\phi')\ne\emptyset$. Indeed, by Lemma \ref{lem:2-intersect}, 
$$
 \gO(\phi(x_1),\ldots,\phi(x_{n-2}))\cap\gO(\phi'(x_1),\ldots,\phi'(x_{n-2}))\ne\emptyset.
$$
Pick $g_n$ in this set. Again, by Lemma \ref{lem:2-intersect},
$$
 \gO(\phi(x_1),\ldots\phi(x_{n-3}),g_n)\cap\gO(\phi'(x_1),\ldots\phi'(x_{n-3}),g_n)\ne\emptyset,
$$
so pick $g_{n-1}$ in this intersection. In a recursive way we define $g_i$ for the lower indices. Defining $\psi$ by $\psi(x_i)=g_i,~i=1,\ldots,n$ we obtain an element $\psi$ which links both $\phi$ and $\phi'$. 

Since $\L(\phi)\cap\L(\phi')$ is open nonempty and contained in $\overline{\aut(F_n)\cdot\phi'}$, we may find $\gs\in\aut(F_n)$ such that $\gs\cdot\phi'\in \L(\phi)\cap\L(\phi')$. Similarly we can find $\gt\in\aut(F_n)$ such that $\gt\cdot\phi\in \L(\phi)\cap\L(\phi')\cap\gs\cdot U$. It follows that $\gs^{-1}\gt\cdot\phi\in U$.

\qed

\section{Ergodicity}

We are now in a position to prove Theorem \ref{thm1}.
Let $\{ x_1,\ldots,x_n\}$ be a generating set of $F_n$.

We will say that an $n$-tuple $(g_1,\ldots,g_n)\in G^n$ is {\it
  strongly redundant} if every $(n-1)$-subtuple generates a dense
subgroup of $G$. We first claim that if $n\ge 3$ then there exists a
strongly redundant $n$-tuple. To see this, start with an arbitrary
$(n-1)$-tuple $(g_1,\ldots,g_{n-1})$ which generates a dense
subgroup. If $k$ is archimedean, by Corollary \ref{cor:arch-inter},
$\ti\gO(g_1,\ldots,g_{n-1})\ne\emptyset$ and the claim follows, using
$(g_1,\ldots,g_{n-1},g)$ for any $g\in\ti\gO(g_1,\ldots,g_{n-1})$.
If $k$ is nonarchimedean, slightly deforming the $g_i,~i\le n-1$ we may assume that they are all nontorsion. Then again $\ti\gO(g_1,\ldots,g_{n-1})\ne\emptyset$;
for $n=3$ this follows from Lemma \ref{lem:2-intersect}, while for $n>3$ from Lemma \ref{lem:tiO-nonempty} since at lease two of the ($\ge 3$)
elements $(g_1,\ldots,g_{n-1})$ are simultaneously elliptic or not.

The set $\mathcal{SR}$ of strongly redundant $n$-tuples is open in $G^n$. We will call a subset of $\mathcal{SR}$ of the form $\prod_{i=1}^nU_i$ a {\it strongly redundant open cube}. 
We shall identify $\Hom(F_n,G)$ with $G^n$ via the map $f\mapsto (f(x_1),\ldots,f(x_n))$. In particular, we shall say that a representation $f\in\Hom(F_n,G)$ is strongly redundant if $(f(x_1),\ldots,f(x_n))$ is a strongly redundant $n$-tuple.

Let $A\subset\mathcal{R}_n(G)$ be a measurable $\aut(F_n)$ almost invariant subset. We wish to show that $A$ is either null or conull. 
Replacing $A$ by the countable intersection
$\cap_{\gs\in\Aut(F_n)}\gs\cdot A$ we may assume that it is precisely invariant rather the almost invariant. 

Let us fix once and for all a strongly redundant open cube $U=\prod_{i=1}^nU_i$. 
Arguing as in the proof of \cite[Theorem 1.6]{Gelander}
one deduces that the intersection of $A$ with $U$ is either
null or conull in $U$. Indeed, assuming the contrary, one derives from
Fubini's theorem that for some index $i_0\in\{1,\ldots,n\}$ and a choice of $u_j\in U_j$ for every $j\ne i_0$, the set 
$$
 \{ u\in U_{i_0}: (u_1,\ldots,u_{i_0-1},u,u_{i_0+1},\ldots,u_n)\in A\}
$$ 
is neither null nor conull in $U_{i_0}$ and hence the set 
$$
 Y=\{ g\in G: (u_1,\ldots,u_{i_0-1},g,u_{i_0+1},\ldots,u_n)\in A\}
$$ 
is neither null nor conull in $G$.
However, since $A$ is invariant under Nielsen transformations, $Y$ is invariant under the left action of the group $\langle u_i,~i\ne i_0\rangle$. But this group is dense and hence acts ergodically on $G$, a contradiction. Thus, up to replacing $A$ by its complement, we may assume that $A\cap U$ is null.

Now let $f\in\mathcal{R}_n(G)$ be an arbitrary redundant representation. Since the action of $\aut(F_n)$ on $\mathcal{R}_n(G)$ preserves the topology and is minimal,
for some $\gs\in\aut(F_n)$ we have $\gs\cdot f\in U$ and hence $\gs^{-1}U$ is an open neighborhood of $f$ which meets $A$ in a null set. 

Since $\mathcal{R}_n(G)$ is homeomorphic to an open subset of $G^n$ it is second countable, and thus can be covered by a countable union of open sets, each meets $A$ in a null set. It follows that $A$ is null.
\qed

 %--------------------------

\section{Nonmixing}
In this section we consider the case of $G=\SL(2,\BC)$ and $G=\SL(2,\BR)$, where
hyperbolic geometry gives us additional structure. In these cases we
show that the action on $\mathcal{R}_n(G)$, in spite of being minimal
and ergodic, is not weakly mixing in a suitable sense. We will
consider the action of $\out(F_n)$ on the space of
$\chi_n(G)=\Hom(F_n,G)/G$, letting
$\overline{\mathcal{R}}_n(G)=\mathcal{R}_n(G)/G$ be the space of
conjugacy classes of redundant representations.

\begin{thm}\label{action on pairs}
The action of $Out(F_n)$ on $\overline{\mathcal{R}}_n(G)$, for $n\ge 3$,  is not weakly mixing. Indeed, the
diagonal action is not ergodic on $\overline{\mathcal{R}}_n(G)\times\overline{\mathcal{R}}_n(G)$, and in fact
there is an open nonempty invariant  subset of  $\overline{\mathcal{R}}_n(G)\times\overline{\mathcal{R}}_n(G)$ on
which $Out(F_n)$ acts properly discontinuously. 
\end{thm}
We begin by recalling some definitions. 

If $X$ is a generating set for $F_n$ and $A$ a set of cyclically
reduced words in $F_n$, 
the {\em Whitehead graph} $Wh(A,X)$ is defined as follows:  The
vertex set of $Wh(A,X)$ is set $X^\pm = \{x,x^{-1}: x\in X\}$. An
(unoriented) edge $[ab]$ appears whenever $ab^{-1}$ is a subword of a cyclic
permutation of a word of $A$ (and in addition $[aa^{-1}]$ is an edge
whenever $A$ contains the length 1 word $a$).  See Whitehead \cite{whitehead:graph,whitehead:equivalence}
and Stallings \cite{stallings:whitehead}.

For a single word write $Wh(\gamma,X) =
Wh(\{\gamma\},X)$.  If $\alpha$ is a collection of loops in the handlebody of genus
$n$, or conjugacy classes in $F_n$, define $Wh(\alpha,X)$ to be
$Wh(a,X)$ for (any) set of cyclically reduced words representing $\alpha$.

Note that $Wh(A\cup B,X) = Wh(A,X) \cup
Wh(B,X)$ where ``union'' of graphs  means union followed by
identification of duplicate edges.

\medskip

An element of $F_n$ is {\em primitive} if it is a member of a free
generating set. Whitehead gave the following property as part of an
algorithm for deciding primitivity in $F_n$:

{\bf Basic Lemma (Whitehead): }   {\em If $\gamma$ is a primitive cyclically
reduced element then, for any generating set $X$,  $Wh(\gamma,X)$ is
either disconnected or has a cutpoint. }

\medskip

\subsection*{Primitive-stable pairs}
Since $G$ acts on $\BH^3$ in both the real and complex case, we can
consider as in \cite{minsky:primitive} the geometric properties of
representations in $Hom(F_n,G)$. 

We define a subset $\PS^2_n\subset \chi_n(G)^2$ as follows. 
Recall from \cite{minsky:primitive} that
for each $\rho\in Hom(F_n,G)$ and basepoint $x\in \BH^3$ there is an
orbit map $\tau_{\rho,x} : F_n \to \BH^3$, namely $g\mapsto\rho(g)x$.  Fixing a set of
generators we also extend $\tau_{\rho,x}$ to the corresponding Cayley
graph of $F_n$, by mapping edges to geodesic segments. 

Recall also that every nontrivial element of $F_n$ has an axis in the
Cayley graph, and let  $\PP$ denote the set  of axes of {\em
  primitive} elements. 

Given a constant $K$ and basepoint $x$, let $\AAA(K,x,\rho)$ denote
the set of axes which $\tau_{\rho,x}$ maps $K$-quasi-geodesically to
$\BH^3$. (A map $f:\BR\to Y$ to a metric space $Y$ is $K$-quasi-geodesic if 
$|s-t|/K - K \le d_Y(f(s),f(t)) \le K|s-t| + K$.)
 In \cite{minsky:primitive}, $\PS_n$ was defined as the
(conjugacy classes of) representations for which there exists $K,x$ such that
$\PP\subset \AAA(K,x,\rho)$.

Now define $\PS^2_n$ as the set of pairs $([\rho_1],[\rho_2])$ such
that there exist representatives $\rho_1,\rho_2$,  $K>0$, and $x\in \BH^3$ with 
$$
\PP \subset \AAA(K,x,\rho_1) \cup \AAA(K,x,\rho_2).
$$

\medskip

We state some basic properties of this set: 
\begin{lem}\label{PS2 properties}
Let $n\ge 3$ and $G=\SL(2,\BC)$ or $\SL(2,\BR)$. 
\begin{enumerate}
\item $\PS^2_n$ is open 
\item $\PS^2_n$ is $Out(F_n)$-invariant
\item the action on $\PS^2_n$ is properly discontinuous.
\end{enumerate}
\end{lem}

\begin{proof} 
The proof proceeds essentially as
in \cite{minsky:primitive}
for the corresponding facts for
$\PS_n$. We give sketches. 

\subsubsection*{Part (1)}
In \cite{minsky:primitive} in the proof of Theorem 3.2, 
the following stability property is given:
\begin{lem}\label{stability}
Given $K, x$ and $\rho$ there exists $K'$ and a neighborhood $U$ of
$\rho$ such that
$$
\AAA(K,x,\rho) \subset \AAA(K',x,\rho')
$$
for all $\rho'\in U$
\end{lem}
The idea of this is the following: Let $A$ be an axis in
$\AAA(K,x,\rho)$. Its $\tau_{\rho,x}$ image is composed of a
biinfinite sequence of segments such
that successive ones are related by (conjugates of) $\rho$-images of
generators. The $K$-quasi geodesic property implies that this axis
makes ``definitely fast'' progress in $\BH^3$, which means the
following: There exists $k$ depending only on $K,x$ and $\rho$ such
that any pair of segments separated by $k$ steps are separated by a
hyperplane in $\BH^3$ such that the sequence of hyperplanes separate
each other and are pairwise separated by a distance strictly greater
than 0. A small perturbation of $\rho$ affects each sequence of $k$
generators by a small amount, and hence preserves this hyperplane
property (but changes the constants). Hence the $\tau_{\rho',x}$ image
of $A$ is $K'$-quasi-geodesic, where $K'$ depends on $K, x$ and how
close $\rho'$ is to $\rho$.

With this lemma in hand, suppose $([\rho_1],[\rho_2])\in \PS^2_n$ and
let $x,K$ be such that $\PP$ is contained in $\AAA(K,x,\rho_1)\union
\AAA(K,x,\rho_2)$. Let $U_1, K_1$ and $U_2,K_2$ be given by Lemma \ref{stability}
for $\rho_1$ and $\rho_2$ respectively, and let $K'=\max(K_1,K_2)$. Then we have
$$
\PP\subset \AAA(K',x,\rho'_1) \union \AAA(K',x,\rho'_2)
$$
for all $(\rho_1',\rho_2')\in U_1\times U_2$. It follows, letting
$U'_i$ denote the image of $U_i$ in $\chi_n(G)$, 
that $U'_1\times U'_2\subset\PS^2_n$. 
(Note that $Hom(F_n,G)\to\chi_n(G)$, being a quotient by a group
action, is an open map, so that  $U'_1$
and $U'_2$ are open.)

\subsubsection*{Part  (2)}
Suppose $([\rho_1],[\rho_2])\in\PS^2_n$. 
Any $\psi\in Aut(F_n)$ acts by quasi-isometry on the Cayley graph of
$F_n$, and it follows that the image of the axis of any $g\in F_n$ is
a quasi-geodesic (with constants depending on $\psi$) that shadows
the axis of $\psi(g)$. Now if $g$ is primitive, so is $\psi(g)$, so
that the axis of $\psi(g)$ is in $\AAA(K,x,\rho_i,)$ for $i=1$ or
$2$. But this means, for $K'$ depending on $K$ and the quasi-isometry
constant of $\psi$, that the axis of $g$ is in
$\AAA(K',x,\rho_i\circ\psi)$. Hence
$([\rho_1\circ\psi],[\rho_2\circ\psi])\in \PS^2_n$ too. 

\subsubsection*{Part (3)}
(following the argument of Theorem 3.3 of
\cite{minsky:primitive})

For a conjugacy class $c\in F_n$ let $||c||$ denote the length of a
cyclically reduced representative, or equivalently the translation
length of any representative of $c$ on its axis in the Cayley graph, and let $\ell_\rho(c)$ denote the
translation length of the conjugacy class $\rho(c)$ in $\BH^3$. If
the axis of (any representative of) $c$ is in $\AAA(K,x,\rho)$ then $\ell_\rho(c)/||c||$ is
bounded above and below by positive constants depending on $K,x$. 

So now if $([\rho_1],[\rho_2])\in\PS^2_n$, all primitive conjugacy
classes $c$ satisfy such a bound either on $\ell_{\rho_1}(c)/||c||$ or 
on $\ell_{\rho_2}(c)/||c||$. Moreover these bounds vary by a bounded
ratio for a fixed $c$ and small perturbations of the representation,
as a consequence of Lemma \ref{stability}.

Now let $E$ be a compact subset in $\PS^2_n$. The above gives us
uniform upper
and lower bounds either on $\ell_{\rho_1}(c)/||c||$ or on
$\ell_{\rho_2}(c)/||c||$, for each primitive $c$, over all of $E$. 
If $\Phi\in Out(F_n)$
such that $\Phi(E)\cap E \ne \emptyset$, let $(\rho_1,\rho_2)$ be in
this intersection. For each primitive $c$ we obtain a bound of the
form $\ell_{\rho_i}(c) \le b_1||c||$ for $i=1$ {\em and} $i=2$, simply
because the maps $\tau_{\rho_i,x}$ have uniform Lipschitz bounds on
$E$. Since $(\rho_1,\rho_2)\in \Phi(E)$ we also obtain a bound of the
form $||\Phi(c)|| \le b_2 \ell_{\rho_i}(c)$, for at least one
$i\in\{1,2\}$. Putting these together we obtain a uniform upper bound on
$||\Phi(c)||/||c||$. This suffices, e.g. by Lemma 3.4 of
\cite{minsky:primitive}, to restrict $\Phi$ to a finite set in
$Out(F_n)$. It follows that the action is  properly discontinuous.

\end{proof}
\medskip

Lemma \ref{PS2 properties} tells us that 
$PS^2_n\cap \overline{\mathcal{R}}_n(G)\times\overline{\mathcal{R}}_n(G)$ is the set
required by Theorem \ref{action on pairs}, provided we can prove that it is non-empty.

The proof of this will take a somewhat different form when
$G=\SL(2,\BC)$ and $G=\SL(2,\BR)$. Although it suffices in fact to
prove the real case since it embeds in the complex case, we give a
separate proof in the complex case since the theory of hyperbolic 3-manifolds can be
applied, giving a more flexible and geometric construction.

\subsection{The complex case}
The proof  will hinge on the following construction:

\begin{lem}\label{pair of graphs}
Let $H$ be the genus $n$ handlebody for $n\ge 3$. 
There exist simple loops  $\alpha_1$ and $\alpha_2$
on  $\partial H$, and a generating set $X$ for $\pi_1(H)$,
such that a representative of each $\alpha_i$ in $F_n$  is contained in a
proper free factor, but
the Whitehead graph of the union, $Wh(\{\alpha_1,\alpha_2\},X)$,
is connected and without cutpoints. 

Moreover, each $\alpha_i$ can be chosen so that $H$ admits a
geometrically finite hyperbolic structure for which $\alpha_i$ is the
unique parabolic. 
\end{lem}

\begin{proof}
We can write $H$ as a 
boundary connected sum (i.e. gluing along disks)  
$$
H = T_1 \cup H' \cup T_2
$$
where $H'$ is a handlebody of genus $n-2$ and $T_1$ and $T_2$ are
handlebodies of genus 1, i.e. solid tori. We then rearrange $H$ as a
union of overlapping handlebodies of genus $n-1$,  $H_1 = T_1 \cup H'$
and $H_2 = T_2 \cup H'$. 

Choose generators $X = \{X_1,\ldots,X_n\}$ for $F_n = \pi_1(H)$ so that $X_1$
generates $\pi_1(T_1)$, $X_2$ generates $\pi_1(T_2)$, and the rest generate
$\pi_1(H')$. Now for $i=1,2$ suppose that $\gamma_i$ is an element of
$\pi_1(H_i)$ whose Whitehead graph 
$Wh(\gamma_i,\{X_i,X_3,\ldots,X_n\})$ is connected and without
cutpoints. Considered with respect to all generators, $Wh(\gamma_1,X)$
is disconnected because it has no edges incident to $X_2^\pm$, and
indeed $\gamma_1$ is contained in the proper free factor $\langle
X_1,X_3,\ldots,X_n\rangle$. 
The corresponding statements hold for $\gamma_2$. 

However, 
$Wh(\{\gamma_1,\gamma_2\},X)$ is the union of two connected graphs
without cutpoints, which together meet every vertex and intersect
along at least two vertices (since there are two vertices per
generator). It follows that $Wh(\{\gamma_1,\gamma_2\},X)$ is both
connected and without cutpoints. 

Now as discussed in \cite{minsky:primitive}, we may select $\alpha_i$
in the {\em Masur domain} of $H_i$, and this will imply that the
Whitehead graph of $\alpha_i$ with respect to some generating set will
be connected without cutpoints. Applying a homeomorphism if necessary,
we may assume that the generating set is the one we have already
fixed. Moreover, being in the Masur domain implies that $H_i$ admits a
geometrically finite hyperbolic structure for which $\alpha_i$ is the
sole parabolic. 

One can always
choose representatives of such curves on $H_1$ or on $H_2$ which are
disjoint from the gluing disks, and hence they can be made to lie on
the boundary of $H$. Finally, the geometrically finite representations
we have on $H_i$ can be extended to representations on $H$ which are
still geometrically finite with the $\alpha_i$ as sole parabolics --
this is an immediate consequence of the Klein Combination Theorem
which gives conditions on constructing free products of Kleinian
groups (in this case, the factors are $\rho_i(\pi_1(H_i))$ and a
hyperbolic cyclic group corresponding to $\pi_1(T_{3-i})$) and
describes the type of the resulting group. See Klein \cite{klein:combination} and
Maskit \cite{maskit:klein}. 
\end{proof}

Let $\alpha_1,\alpha_2$ and the generating set $X$ be as in Lemma
\ref{pair of graphs}, and let $\rho_1, \rho_2:\pi_1(H)\to G$ be representations
corresponding to the geometrically finite structures the lemma
provides for $\alpha_1$ and $\alpha_2$. 
We claim that $([\rho_1],[\rho_2])\in PS^2_n$.  The proof follows the
argument in \cite{minsky:primitive} with minor variations:

The property of geometric finiteness implies that each quotient
manifold $N_i = \BH^3/\rho_i(F_n)$ contains a convex 
core $C_i$ which is not compact, but can be written as $H_i \cup
Q_i$ where $H_i$ is a compact handlebody and $Q_i$ is a cusp
neighborhood associated to $\alpha_i$. All closed geodesics in $N_i$
are contained in $C_i$, and if a closed geodesic $\gamma$ penetrates
deeply into $Q_i$ then the corresponding reduced word in $F_n$
contains a high power of the reduced form of $\alpha_i$. This is shown
in the proof of Theorem 4.1 of \cite{minsky:primitive}. 

Hence, after possibly enlarging $H_i$, we have the following property: If a
conjugacy class $\gamma$ in $F_n$ has geodesic representative
$\gamma^i$ which is {\em not} contained in $H_i$, then $Wh(\gamma,X)$
contains $Wh(\alpha_i,X)$. On the other hand by compactness there
exists $K$ and $x\in \BH^3$ so that if $\gamma^i$ is contained in
$H_i$ then the axis of $\gamma^i$ is mapped $K$-quasi-geodesically by
$\tau_{\rho_i,x}$.  Let $P_i$ be the set of conjugacy classes whose
geodesic representatives in $N_i$ are contained in $H_i$, and hence
satisfy the quasi-geodesic condition with respect to $\rho_i$.

What we have shown is that any element {\em not} in $P_1\cup P_2$
has Whitehead graph containing $Wh(\{\alpha_1,\alpha_2\},X)$. Since
this graph is connected and without cutpoints, Whitehead's Lemma tells
us that such an  element cannot
be primitive. We conclude that $P_1\cup P_2$ cover all the primitive
elements, so that $([\rho_1],[\rho_2])\in \PS_n^2$.

Now, since $\alpha_i$ is contained in a proper free factor $B_i<F_n$,
$\rho_i|_{B_i}$ is not a Schottky group. It can therefore be
approximated by dense representations of $B_i$ (as in the proof of
Lemma 3.2 of \cite{minsky:primitive}). It follows that $\rho_i$ can be
approximated by redundant representations of $F_n$. Since $\PS^2_n$ is
open, we conclude that $\PS^2_n\cap
(\overline{\mathcal{R}}_n(G)\times\overline{\mathcal{R}}_n(G))$ is nonempty.

This concludes the proof in the complex case. 

\subsection{The real case}
A discrete faithful representation  $F_n\to\SL_2(\BR)$ corresponds to a
Fuchsian group, and if this group is a lattice with just one parabolic then the
representation is automatically in $\PS_n$, by the main theorem of 
\cite{minsky:primitive}. So we have to consider groups with two
or more parabolics.

Let $\Sigma$ be a sphere with $k \ge 4$ punctures. Then $\pi_1(\Sigma)$
can be written as a free group on $n = k-1$ letters $X_1,\ldots,X_n$,
representing $k-1$ of the punctures, with the last puncture
represented by the product $X_1X_2\cdots X_n$.

Let $g_1$ be a cyclically reduced word in the generators
$X_2,\ldots,X_n$, such that the
Whitehead graph $W_1 = Wh(g_1,\{X_2,\ldots,X_n\})$ is connected and
without cutpoints. Let $\Phi_1\in \aut(F_n)$ be the automorphism
defined by:
\begin{align*}
X_1 &\mapsto X_1 g_1^m \\
X_2 &\mapsto X_2 X_1 g_1^m\\
& \cdots \\
X_n &\mapsto X_n X_1 g_1^m.\\
\end{align*}
(This can be obtained as a composition of Nielsen moves, first
multiplying $X_1$ by the letters in $g_1^m$, and then multiplying each
$X_i$ for $i>1$ by the image of $X_1$.)
If $m$ is chosen sufficiently large then the $\Phi_1$ image of each of
the $k$ punctures has Whitehead graph (with respect to all the
generators) containing $W_1$.

Let $\rho_0:\pi_1(\Sigma) \to \SL_2(\BR)$ be a discrete faithful Fuchsian
representation
taking all $k$ punctures to parabolics, and let 
$$
\rho_1 = \rho_0 \circ \Phi_1^{-1}.
$$
Then the parabolics of $\rho_1$ are in $k$ conjugacy classes, each of
whom by itself has Whitehead graph containing $W_1$. 

Now define $g_2$, $W_2$,  $\Phi_2$ and $\rho_2$ the same way, but
interchanging the roles of $X_1$ with $X_2$.
The graph $W_2$ is then connected and without
cutpoints when restricted to the vertices associated to
$X_1,X_3,\ldots,X_n$.

The rest of the proof goes through in essentially the same way to show
that $([\rho_1],[\rho_2])$ is in $\PS^2_n$. Namely, a conjugacy class
whose axis is badly non-quasi-geodesic in both representations must
wrap around at least one of the parabolics in $\rho_1$ and at least one of the
parabolics in $\rho_2$ as well. Hence its Whitehead graph (with
respect to all $n$ generators) contains
$W_1\cup W_2$. Since $W_1$ and $W_2$ intersect in at least two
vertices (those associated to $X_3$), their union is connected and without cutpoints. 

We see that each $\rho_i$ is approximated by redundant representations as
before, since each contains a parabolic that is inside a free factor
(in fact is itself primitive).

\begin{rem}\label{rem:NMr1}
The non-mixing result Theorem \ref{thm2} extends to every rank one simple Lie group. Indeed, Lemma \ref{PS2 properties} holds in this generality and the same proof applies. The only issue that requires some justification is the non-emptiness of $\PS^2_n\cap\overline{\mathcal{R}}^2$. However, since every simple Lie group $G$ admits a subgroup $H$ locally isomorphic to $\SL_2(\BR)$ such that for some point $x$ in the symmetric space $G/K$ the orbit $H\cdot x$ is isometric to the hyperbolic plane $\BH^2$ (cf. \cite[Theorem 3.7]{Pla-Rap}), this result can be deduced easily from the $\SL_2(\BR)$ analog.
\end{rem}

\section{Some related problems}
Let us end this paper by recalling and suggesting some old and new related problems.

\subsection{The other conjecture of Lubotzky}

First let us repeat Lubotzky's second conjecture \cite{Lub}, mentioned also in \cite{minsky:primitive}, which is still a mystery, even for $\SL_2(\BR)$ and $\SL_2(\BC)$:

\begin{prob}\label{PS+R}
Let $n\ge 3$.
Given a connected simple Lie group, is it true that almost every representation of $F_n$ is either redundant or primitive stable? 
\end{prob}

When $k$ is non-archimedean and $G=\BG(k)$ is the group of $k$
rational points of some Zariski connected simple algebraic group
$\BG$, Problem \ref{PS+R} still makes sense when restricting to {\it
  unbounded} representations. It can be deduced from \cite{Glasner}
that for $\SL_2(k)$ almost every dense representation of $F_n$ is
redundant. Hence, the question in this case is whether almost every
discrete faithful representation is primitive stable or even
Schottky. For higher rank groups, e.g. for $\SL_3(k)$ it is unclear
if the definition of primitive stable representations extends in a
useful way.

\subsection{Does density of primitives imply redundant?}
It is straightforward, that if $f:F_n\to G$ is redundant then $f(P_n)$
is dense in $G$, where $P_n\subset F_n$ is the set of primitive
elements. Moreover, if $G$ is discrete (e.g. finite) then the opposite
is also true, i.e. if $f(P_n)=G$ then $f$ is redundant (consider a
basis containing a primitive element that maps to $1\in G$).

\begin{prob}
Let $G$ be a topological group, or more specifically, a simple Lie group over a local field. Is it true that every representation $f:F_n\to G$ for which $f(P_n)$ is dense, is redundant?
\end{prob}

\subsection{Extending the results of this paper to semisimple groups}
In this paper we restricted ourselves to the case where the group $G$ is simple rather than semisimple. However, as remarked in \ref{rem:semisimple}, over $\BR$ or $\BC$, Theorems \ref{thm1} and \ref{thm3} remain true, with only slight modifications in the arguments, when $G$ has more than one factor. However for nonarchimedean fields, although we expect that the theorems remain valid, some parts of our proofs do not directly apply:

\begin{prob}
Extend the result of this paper to all semisimple groups over local fields, and more generally, to groups of the form $\prod_{i=1}^n\BG_i(k_i)$ where $\BG_i$ are simple algebraic groups and $k_i$ are local fields.
\end{prob}

\subsection{Other notions of weak mixing}

For a group $G$ acting on a finite measure space $X$, weak mixing is equivalent to each of the following:

\begin{enumerate}
\item the action of $G$ on $X\times X$ is ergodic,
\item for every finite measure preserving ergodic space $Y$, the action on $X\times Y$ is ergodic,
\item the unitary representation $L^2(X)$ has no finite dimensional sub-representation but the constants.
\end{enumerate}

For a compact Lie group $G$, since the space $\mathcal{R}_n(G)=\Hom(F_n,G)$ has finite measure, the three conditions above are equivalent and, 
as shown in \cite{Gelander}, satisfied whenever $n\ge 3$.  

However, when $G$ is noncompact and $n\ge 3$, the space $\mathcal{R}_n(G)$ is an infinite measure space, hence the various notions of {\it weak mixing} are no longer equivalent. One may still ask whether for every ergodic probability space $Y$ the action of $\aut(F_n)$ on $\mathcal{R}_n(G)\times Y$ is ergodic. In particular, one may study this question in the special case $Y=\Hom(F_n,H)$ where $H$ is a connected compact group.

\subsection{The notion of {\it Spread} for topological groups}
Recall that a finite (or discrete) group $G$ is said to have {\it
  spread} $k$ if for any $k$ nontrivial elements $g_1,\ldots,g_k\in
G\setminus \{1\}$ there is $h\in G\setminus \{1\}$ which generates $G$
simultaneously with each of the $g_i$'s, i.e. $\forall i, \langle
h,g_i\rangle=G$. Any finite simple group has spread $2$ (see
\cite{GKS,GS}). We can extend the definition of spread to topological
groups by requiring that the closure $\overline{\langle h,g_i\rangle}$
equal $G$ for all $i$. 

Let now $G$ be a connected center-free simple Lie group. By \cite{AV} $G$ has spread $1$ (see also \cite{Ku1,Ku2}). Additionally, 
given $g\in G$, it is not hard to show that if $\langle g,h\rangle$ is
dense in $G$ for some $h\in G$ then the set $\{h\in
G:\overline{\langle g,h\rangle}=G\}$ contains a neighborhood of the
identity minus some exponential proper subvariety. Hence if $G$ has spread $1$, it has spread $k$ for any finite $k$. 
The same result holds with respect to Zariski topology.

Similarly, one can define the notion of {\it random-spread} as follows: Say that $G$ has random-spread $k$ if for almost every $k$ elements $g_1,\ldots,g_k$ in $G$ there is a simultaneous generating partner, i.e. $h$ such that $\langle h,g_i\rangle$ is dense for each $i$. One can deduce from the discussion in Section \ref{sec:dense} that every connected semisimple Lie group has random-spread $k$ for every finite $k$. 

It might be interesting however to study the notion of random-spread (as well as the true spread and other variants of it) for semisimple Lie groups over non-Archimedean local fields.

\end{document}